\documentclass[11pt]{amsart}
\usepackage{amsmath,amsthm,amssymb,latexsym,euscript}
\usepackage{mathrsfs}
\usepackage{amsfonts}
\usepackage{amscd}
\usepackage[all]{xypic}

\let\goth\mathfrak

\def\gc{\goth c}

\def\gu{\goth u} \def\fru{{\gu}}
\def\frv{\goth v}

\def\gz{\goth z}

\def\uce{\mathfrak{uce}}
\def\frsp{\mathfrak{sp}}
\def\sq{\mathfrak{sq}}
\def\fre{\goth e}

\newcommand\scQ{\mathcal Q}
\newcommand\scF{\mathcal F}

\newcommand\ch{\sp{\scriptscriptstyle\vee}}
\newcommand\ot{\otimes}

\def\beq{\begin{equation}}
\def\eeq{\end{equation}}
\def\bea{\begin{eqnarray}}
\def\eea{\end{eqnarray}}
\def\beas{\begin{eqnarray*}}
\def\eeas{\end{eqnarray*}}

\def\cplus{\hbox{$\supset${\raise1.05pt\hbox{\kern -0.55em
${\scriptscriptstyle +}$}}\ }}

 \DeclareMathOperator{\tr}{tr}
\DeclareMathOperator{\str}{str} 
 
\DeclareMathOperator{\Span}{Span}

 \DeclareMathOperator{\Id}{Id}

\DeclareMathOperator{\supp}{supp} 
\DeclareMathOperator{\Ker}{Ker}

\DeclareMathOperator{\Mat}{Mat}
\DeclareMathOperator{\osp}{\mathfrak{osp}}

\DeclareMathOperator{\HC}{HC} \DeclareMathOperator{\rmH}{H}
\DeclareMathOperator{\HF}{HF}
\def\N{\mathbb N} 
\def\Z{\mathbb Z} \def\ZZ{\Z}

\newcommand\al{\alpha}
\newcommand\be{\beta}

\newcommand\la{\lambda} 
 \newcommand\vphi{\varphi}

\newcommand\ta{\tau}

\def\dirlim{\varinjlim}

\def\la{\langle}
\def\ra{\rangle}
\newcommand\lsl{\ensuremath{\mathfrak{sl}}}
\newcommand\lgl{\ensuremath{\mathfrak{gl}}}
\newcommand\st{\ensuremath{\mathfrak{st}}}

\theoremstyle{plain}
\newtheorem{theorem}{Theorem}[section]
\newtheorem{lemma}[theorem]{Lemma}
\newtheorem{prop}[theorem]{Proposition}
\newtheorem{cor}[theorem]{Corollary}

\theoremstyle{definition}

\newtheorem{example}[theorem]{Example}

\newtheorem{blank}[theorem]{}
\newtheorem{remark}[theorem]{Remark}

\numberwithin{equation}{section} \allowdisplaybreaks

\title{Universal Central Extensions of Direct Limits of Lie
Superalgebras}
\author{Erhard Neher}
\address{E.~Neher: University of Ottawa, Ottawa, Ontario, Canada}
\email{neher@uottawa.ca}
\author{Jie Sun}
\address{J.~Sun: Department of Mathematics \\
University of California, Berkeley, CA 94720, USA}
\email{jiesun@math.berkeley.edu} 
\thanks{This research is partially supported by the Natural
Sciences and Engineering Research Council (NSERC) of Canada through
the first author's Discovery grant and the second author's
Postdoctoral Fellowship.}
\date{}

\begin{document}
\maketitle

\begin{abstract} We show that the universal central extension of a
direct limit of perfect Lie superalgebras $L_i$ is (isomorphic to)
the direct limit of the universal central extensions of $L_i$. As an
application we describe the universal central extensions of some
infinite rank Lie superalgebras.
\end{abstract}

\section*{Introduction}

Central extensions appear naturally in the theory of infinite
dimensional Lie algebras. For example, they are fundamental for the
theory of affine Kac-Moody Lie algebras and extended affine Lie
algebras. Centrally extended Lie algebras often have a more
interesting representation theory than the original Lie algebra,
which makes central extension an interesting topic for applications,
e.g., in physics. A convenient way to find ``all'' of them, is to
determine the universal central extension of a given Lie algebra,
which exists for perfect Lie algebras (well-known) and superalgebras
\cite{N1}.

Direct limit of Lie superalgebras are an important way to construct
infinite dimensional Lie superalgebras. Examples include various
types of locally finite Lie (super)algebras \cite{BB, DP,
penkov:2004, PS}, locally extended affine Lie algebras
\cite{morita-yoshii,neeb:LEALA,n:persp} and Lie superalgebras graded
by locally finite root systems \cite{n:3g,gn2}. These types of Lie
algebras and Lie superalgebras have been intensively studied by many
authors, many more than we have quoted, yet no general results seem
to be known about their universal central extensions, besides the
paper \cite{S} in which the author studied a rather special case,
described in Remark~\ref{rem:S}.

In this paper, we consider the universal central extensions of
general direct limits of Lie superalgebras over an arbitrary base
superring. We show in Theorem~\ref{ucedirlim} that the universal
central extension of a direct limit $\dirlim L_i$ of perfect Lie
superalgebras $L_i$ is canonically isomorphic to the direct limit of
the universal central extensions of $L_i$. This result is new even
for the case of Lie algebras. Crucial for its proof is the fact
(\cite{N1}) that one has a endo-functor $\uce$ on the category of
all Lie superalgebras which gives the universal central extension
for perfect Lie superalgebras.

As an application, we describe in \S\ref{sec:2} the universal
central extensions of some direct limit Lie superalgebras, namely
$\lsl(I; A)$ for $|I|\ge 5$ and $A$ an associative superalgebra
(Proposition~\ref{newassprop}, Corollary~\ref{sl2cor}), $\osp(I;A)$
for $A$ commutative associative (Example~\ref{exam:osp}), locally
finite Lie superalgebras (Example~\ref{exam:lfls}) and Lie algebras
graded by locally finite root systems (Example~\ref{exam:rg}). These
applications are possible since one knows the universal central
extension of the Lie superalgebras over which we take the direct
limit.

{\it Acknowledgements.} The authors thank I. Dimitrov who asked one
of us a question about the universal central extensions of certain
locally finite Lie algebras, now answered in Example~\ref{exam:lf}
and pointed out the reference \cite{S}. The first author
gratefully acknowledges helpful discussions with N.~Lam on the topic
of the paper. We also thank V. Serganova for very useful comments on an earlier version of the paper.

\section{Universal central extensions of direct limits of Lie superalgebras:
General results} \label{sec:1}

\begin{blank}\textbf{Review of universal central extensions of
Lie superalgebras.} \label{rev:uce} Throughout this section we
consider Lie superalgebras $L$ over a commutative superring $S$ as
defined in \cite{N1}. Thus $S$ is an associative, unital
$\ZZ/2\ZZ$-graded ring which is commutative in the sense that $s_1
s_2 = (-1)^{|s_1||s_2|}s_2 s_1$ holds for all homogeneous $s_i \in
S$. Here and in the following $|s|$ denotes the degree of a
homogeneous element. Formulas involving the degree function are
supposed to be valid for homogeneous elements -- a condition that we
will not mention explicitly in the following.

We first describe some facts on central extensions which are needed
in the following. Proofs can be found in \cite{N1}. A central
extension of $L$ is an epimorphism $f: K \to L$ of Lie superalgebras
with the property that $\Ker f \subset \gz(K)$, the centre of $K$. A
central extension $f: K\to L$ is called universal if for any other
central extension $f' : K'\to L$ there exists a unique Lie
superalgebra morphism $g: K \to K'$ such that $f = f' \circ g$. A
universal central extension of $L$ exists and is then unique up to a
unique isomorphism if and only if $L$ is perfect. To describe a
model of a universal central extension of $L$ one can use the
following construction of a Lie superalgebra which is  valid for
any, not necessarily perfect Lie superalgebra $L$.

Let $\mathcal{B}=\mathcal{B}_L$ be the $S$-submodule of the
$S$-supermodule $L\otimes_S L$ spanned by all elements of type
\begin{align*}
  x &\otimes y+ (-1)^{|x| |y|}  y\otimes x, \qquad
     x_{\bar{0}} \otimes x_{\bar{0}} \quad \text{for } x_{\bar{0}}\in L_{\bar{0}}, \\
 (-1)^{|x||z|}x &\otimes [y,z]+(-1)^{|y||x|}y\otimes
          [z,x]+(-1)^{|z||y|}z\otimes [x,y],
\end{align*}
and put
$$\uce(L)=(L\otimes_S L)/\mathcal{B}
 \quad \text{and} \quad \la x,y \ra=x\otimes y +\mathcal{B}\in \uce(L).$$
The supermodule $\uce(L)$ becomes a Lie superalgebra over $S$ with
respect to the product
$$
   \big[ \la l_1,l_2 \ra , \, \la l_3, l_4\ra \big] = \big\la
      [l_1,l_2],\, [l_3,l_4]\big\ra
$$
for $l_i \in L$. The map
\begin{equation} \label{rev:uce1}
  \mathfrak{u}=\mathfrak{u}_L:\uce(L)\rightarrow L: \quad \la x,y\ra \mapsto [x,y]
\end{equation} is a Lie superalgebra morphism with kernel
$\Ker \fru\subset \gz(\uce(L))$. If $L$ is perfect, then
$\mathfrak{u}:\uce(L)\rightarrow L$ is a universal central extension
of $L$. A morphism of Lie superalgebras $f:L\rightarrow M$ gives
rise to a morphism of Lie superalgebras
$$\uce(f):\uce(L)\rightarrow \uce(M): \quad \la l_1,l_2\ra \mapsto
    \la f(l_1),f(l_2)\ra.$$
The assignments $L \mapsto \uce(L)$ and $f \mapsto \uce(f)$ define a
covariant endo-functor on the category $\mathbf{Lie}_S$ of Lie
$S$-superalgebras.

Similar to Lie algebras, a central extension of a Lie
$S$-superalgebra $L$ can be constructed by using a $2$-cocycle $\ta
: L \times L \to C$. Here $C$ is a $S$-supermodule, $\ta$ is
$S$-bilinear of degree $0$ whence $\ta(L_\al, L_\be) \subset
C_{\al+\be}$ for $\al,\be \in \ZZ/2\ZZ$, alternating in the sense
that $\ta(x, y) + (-1)^{|x||y|} \ta (y,x) = 0= \ta(x_{\bar 0},
x_{\bar 0})$ for $x_{\bar 0} \in L_{\bar 0}$, and satisfies
$$(-1)^{|x||z|} \ta(x, [y,z]) +(-1)^{|y||x|}\ta(y, [z,x])
+(-1)^{|z||y|}\ta(z, [x,y])=0.$$ Equivalently, a $2$-cocycle is a
map $\ta : L \times L \to C$ such that $L\oplus C$ is a Lie
superalgebra with respect to the grading $(L\oplus C)_\al = L_\al
\oplus C_\al$ and product $[l_1\oplus c_1, l_2\oplus
c_2]=[l_1,l_2]_L\oplus \tau(x_1,x_2)$ where $[.,.]_L$ is the product
of $L$. In this case, the canonical projection $L\oplus C \to L$ is
a central extension.
\end{blank}

\begin{blank} \textbf{Review of direct limits.} We recall some notions regarding
direct limits.  Let $(I, \leq)$ be a directed set, which will be
fixed throughout this section. A {\it directed system\/} is a family $(L_{i} :
i \in I)$ in $\mathbf{Lie}_S$ together with Lie superalgebra
morphisms $f_{ji}: L_{i}\rightarrow L_{j}$
for every pair $(i,j)$ with $i \le j$ such that 
$f_{ii}=\Id_{L_{i}}$ and $f_{ki}=f_{kj}\circ f_{ji}$ for $i\leq
j\leq k$. A {\it direct limit\/} of the directed system
$(L_{i},f_{ji})$ is a Lie superalgebra $L$ 
together with Lie superalgebra morphisms $\vphi_{i}:L_{i}\rightarrow
L$ satisfying 
$\vphi_{i}=\vphi_{j}\circ f_{ji}$, and 
for any other such pair $(Y,\psi_{i})$, i.e., $\psi_i = \psi_j \circ
f_{ji}$ for $i\le j$, there exists a unique morphism
$\varphi:L\rightarrow Y$ such that the following diagram commutes.
$$\xymatrix{
 L_{i} \ar[rr]^{f_{ji}} \ar[dr]^{\vphi_{i}}\ar@/_/[ddr]_{\psi_{i}} &
   & L_{j} \ar[dl]_{\vphi_{j}}\ar@/^/[ddl]^{\psi_{j}}\\
    &L \ar@{-->}[d]^{\varphi}& \\
& Y&}
$$ The usual construction of a direct limit of modules shows
that a direct limit of Lie superalgebras exists in $\mathbf{Lie}_S$
and is unique, up to a unique isomorphism. We can therefore speak of
``the'' direct limit, and follow the usual abuse of notation and
denote a direct limit of $(L_i, f_{ji})$ by $\varinjlim L_i$. We
will call $\vphi_i$ the {\it canonical maps}.

Let $(K_{i},g_{ji})$ and $( L_{i},f_{ji})$ be two directed systems
of Lie superalgebras, both indexed by the directed set $I$. A
\textit{morphism} from $( K_{i},g_{ji})$ to $( L_{i},f_{ji})$ is a
family $(h_{i} : i \in I)$ of Lie superalgebra morphisms
$h_{i}:K_{i} \rightarrow L_{i}$ such that for all pairs $(i,j)$ with
$i\leq j$ the diagram
$$\xymatrix{K_{i}\ar[r]^{g_{ji}}\ar[d]_{h_{i}}&K_{j}\ar[d]^{h_{j}}\\
L_{i}\ar[r]^{f_{ji}}&L_{j}}$$ commutes. A morphism from $( K_{i},
g_{ji})$ to $( L_{i}, f_{ji})$ gives rise to a unique Lie
superalgebra morphism
$$h = \dirlim h_i :\dirlim K_{i}\longrightarrow\dirlim L_{i}$$
such that $h\circ\vphi_{i}=\psi_{i}\circ h_{i}$ for all $i\in I$,
where $\vphi_{i}: K_{i}\rightarrow \dirlim K_{i}$ and $\psi_{i}:
L_{i}\rightarrow \dirlim L_{i}$ are the canonical maps. Since direct limits preserve exact
sequences \cite[II, \S6.2, Prop.~3]{bou:A}, it follows that $h$ is
injective (respectively surjective) if all $h_{i}$ are injective
(respectively surjective).
\end{blank}

\begin{blank}\label{setti} Let $(L_{i},f_{ji})$ be a directed system
of Lie superalgebras in $\mathbf{Lie}_S$ and let $\dirlim  L_{i}$ be
its direct limit with canonical maps
$\vphi_{i}:L_{i}\rightarrow\dirlim L_i$. Since $\uce$ is a covariant
functor, it is immediate that $(\uce(L_{i}),\uce(f_{ji}))$ is also a
directed system of Lie superalgebras. We abbreviate $\uce(f_{ji})$
by $\widehat{f}_{ji}$, and let
$\widetilde{\vphi}_{i}:\uce(L_{i})\rightarrow\dirlim \uce(L_{i})$ be
the canonical maps into the direct limit of $(
\uce(L_{i}),\widehat{f}_{ji})$. \begin{equation}
  \label{comm}
\xymatrix{L_{i} \ar[rr]^{f_{ji}} \ar[dr]_{\vphi_{i}}& &
       L_{j} \ar[dl]^{\vphi_{j}}\\
     &\dirlim  L_{i} }
\quad \xymatrix{ \\ \rightsquigarrow\\ }\quad
 \xymatrix{\uce(L_{i}) \ar[rr]^{\widehat{f}_{ji}}
    \ar[dr]_{\widetilde{\vphi}_{i}}& &
       \uce(L_{j}) \ar[dl]^{\widetilde{\vphi}_{j}}\\
     & \dirlim \uce(L_i) }
\end{equation} Let $\fru_i : \uce(L_i) \to L_i$ be the Lie
superalgebra morphism of \eqref{rev:uce1}. By construction of the
maps $\widehat{f}_{ji}$, we have a commutative diagram
\begin{equation}\label{uhom}
\vcenter{
\xymatrix{\uce(L_{i})\ar[r]^{\widehat{f}_{ji}}\ar[d]_{\mathfrak{u}_{i}}
     &\uce(L_{j})\ar[d]^{\mathfrak{u}_{j}}\\
        L_{i}\ar[r]^{f_{ji}}&L_{j}}
} \end{equation} for $i \le j$. In other words, the family $(\fru_i,
i\in I)$ is a morphism from the directed system $(\uce(L_i),
\widehat{f}_{ji})$ to the directed system $(L_i, f_{ji})$, and
therefore gives rise to a morphism \begin{equation} \label{revdir1}
  \dirlim \fru_i : \dirlim \uce(L_i) \to  \dirlim L_i. \end{equation}
\end{blank}

\begin{lemma}\label{cedirlim}In the setting of\/ {\rm \ref{setti}},
the map \eqref{revdir1} has central kernel, and is a central
extension if all  $L_{i}$ are perfect.
\end{lemma}

\begin{proof} To prove that $\frv := \dirlim
\fru_i$ has central kernel,  let $x\in \Ker \frv$. Thus
$x=\widetilde{\vphi}_{j}(x_{j})$ for some $x_{j}\in \uce(L_{j})$ and
$0 = \frv(x)=\vphi_j(\mathfrak{u}_{j}(x_{j})) $ in $L=\dirlim L_i$.
Hence there exists $k\geq j$ such that
$f_{kj}(\mathfrak{u}_{j}(x_{j}))=0\in L_{k}$. Note
$\vphi_k(f_{kj}(\mathfrak{u}_{j}(x_{j})))=0\in L$. For any $y\in L$,
we have to show that $[x,y]=0$ in $L$. We have
$y=\widetilde{\vphi}_{p}(y_{p})$ for some $y_{p}\in \uce(L_{p})$.
For the above $k,p\in I$ there exists $q\in I$ such that $q\geq
k\geq j$ and $q\geq p$. Thus
$f_{qj}(\mathfrak{u}_{j}(x_{j}))=(f_{qk}\circ
f_{kj})(\mathfrak{u}_{j}(x_{j}))=0\in L_{q}$. The commutative
diagram \eqref{uhom} for $j\le q$ now implies
$\widehat{f}_{qj}(x_{j})\in \Ker \mathfrak{u}_{q}\subset
\goth{z}(\uce(L_{q}))$. So we have
$[\widehat{f}_{qj}(x_{j}),\widehat{f}_{qp}(y_{p})]_{\uce(L_{q})}=0\in\uce(L_{q})$
and hence
\begin{eqnarray*}
[x,y]_{\dirlim \uce(L_{i})}
&=&[\widetilde{\vphi}_{j}(x_{j}),\widetilde{\vphi}_{p}(y_{p})]
 _{\dirlim \uce(L_{i})}\\
 &=&\widetilde{\vphi}_{q}([\widehat{f}_{qj}(x_{j}),
\widehat{f}_{qp}(y_{p})]_{\uce(L_{q})})=0.
\end{eqnarray*}
Thus $\Ker \frv \subset \goth{z}(\dirlim \uce(L_{i}))$. If all $L_i$
are perfect, every $\fru_i$ is surjective, and hence so is $\frv $,
proving that $\frv$ is a central extension.
\end{proof}

\begin{blank} We continue with the setting of \ref{setti}, but assume that
every $L_{i}$ is perfect. Then $L=\dirlim L_i$ is perfect too and
therefore has a universal central extension $\fru : \uce(L)\to L$.
Our goal is to prove that the central extension \eqref{revdir1} is a
universal central extension of $L$. By the construction of $L$, the
canonical maps $\vphi_i:L_i\rightarrow L$ are Lie superalgebra
morphisms. We therefore get a unique Lie superalgebra morphism
$\widehat{\vphi}_i:\uce(L_i)\rightarrow \uce(L)$ such that the
following diagram commutes
\begin{equation}\label{uhomL}
\vcenter{
\xymatrix{\uce(L_{i})\ar[r]^{\widehat{\vphi}_{i}}\ar[d]_{\mathfrak{u}_{i}}
     &\uce(L)\ar[d]^{\mathfrak{u}}\\
        L_{i}\ar[r]^{\vphi_{i}}&L}
} \end{equation} where $\fru_i$ and $\fru$ are universal central
extensions of $L_i$ and $L$ respectively. Applying the covariant
functor $\uce$ to the left commutative diagram in \eqref{comm} shows
that $\hat \vphi_i = \uce(\vphi_i) = \uce(\vphi_j \circ f_{ji}) =
\uce(\vphi_j) \circ \uce(f_{ji}) = \hat \vphi_j \circ \hat f_{ji}$.
Thus the outer triangle in the diagram below commutes. Hence, by the
universal property of $\dirlim \uce(L_{i})$, there exists a unique
Lie superalgebra morphism $\varphi:\dirlim \uce(L_{i})\rightarrow
\uce(L)$ such that all triangles commute. \begin{equation}
\label{ucediag1} \vcenter{ \xymatrix{\uce(L_{i})
\ar[rr]^{\widehat{f}_{ji}} \ar[dr]^{\widetilde{\vphi}_{i}}
   \ar@/_/[ddr]_{\widehat{\vphi}_{i}} & & \uce(L_{j})
   \ar[dl]_{\widetilde{\vphi}_{j}}\ar@/^/[ddl]^{\widehat{\vphi}_{j}}\\
&\dirlim \uce(L_{i}) \ar@{-->}[d]^{\varphi}& \\
& \uce(L)&} } \end{equation}
\end{blank}

For the next theorem we define $\rmH_2(L)$ for a perfect Lie
superalgebra $L$ as the kernel of $\fru : \uce(L) \to L$. In case
$L$ is a perfect Lie algebra over a ring $S$ it is known that
$\rmH_2(L)$ is the second homology group of $L$ with trivial
coefficients.

\begin{theorem}\label{ucedirlim} Assume that all Lie superalgebras $L_{i}$
are perfect. Then the map $$\vphi: \dirlim \uce(L_i) \to
\uce(\dirlim L_i)$$ of \eqref{ucediag1} is an isomorphism of Lie
superalgebras, and hence $\dirlim \fru_i : \dirlim \uce(L_i) \to
\dirlim L_i$ is a universal central extension. In
particular, $\dirlim \fru_i$ induces an isomorphism
\begin{equation} \label{dirlimh2}
 \dirlim \rmH_2(L_i)  \cong \rmH_2 (\dirlim L_i).
\end{equation}
\end{theorem}

\begin{proof} We have already noted that $L$ is perfect and therefore
has a universal central extension $\fru : \uce(L)\to L$. By Lemma
\ref{cedirlim} we know that  $\frv=\dirlim \fru_i : \dirlim
\uce(L_i) \to \dirlim L_i$ is a central extension. Thus the
universal property of $\uce(L)$ implies that there exists a unique
Lie superalgebra morphism $\psi:\uce(L)\rightarrow \dirlim
\uce(L_{i})$ such that the following diagram commutes.
$$\vcenter{\xymatrix{
   \uce(L) \ar[rr]^{\psi} \ar[dr]_ {\fru}
     & & \dirlim \uce(L_{i}) \ar[dl]^\frv \\
    &L&}}$$
We claim that $\psi\circ \varphi=\rm{Id}_{\dirlim \uce(L_{i})}$ and
$\varphi\circ \psi=\rm{Id}_{\uce(L)}$. For the proof of these two
equations, the following diagram may be helpful. \[
\xymatrix{
    \uce(L_i) \ar[r]^{\widetilde \vphi_i} \ar[d]_{\fru_i}
         \ar@/^1.5pc/[rr]^{\widehat \vphi_i}
       & \dirlim \uce(L_i) \ar@<0.5ex>[r]^\vphi \ar[dr]_\frv
        & \uce(L) \ar@<0.5ex>[l]^<<<<<\psi \ar[d]^\fru \\
   L_i \ar[rr]_{\vphi_i} && L
 }\] By the universal property of $\dirlim \uce(L_{i})$, in order to
show $\psi\circ \varphi=\rm{Id}_{\dirlim \uce(L_{i})}$, we only need
to check
$(\psi\circ\vphi)\circ\widetilde{\vphi}_i=\widetilde{\vphi}_i$.
Since $\vphi\circ\widetilde{\vphi}_i=\widehat{\vphi}_i$, we are left
to check $\psi\circ \widehat{\vphi}_i=\widetilde{\vphi}_i $ and this
is true by the observation $\frv\circ
\psi\circ\widehat{\vphi}_i=\fru\circ \widehat{\vphi}_i=\vphi_i \circ
\fru_i=\frv\circ \widetilde{\vphi}_i$ and the uniqueness in
\cite[Prop.~1.13]{N1}. For the proof of $\varphi\circ
\psi=\rm{Id}_{\uce(L)}$ it is in view of the universal property of
$\uce(L)$ enough to verify $\fru \circ (\vphi\circ \psi)=\fru$.
Since $\fru=\frv\circ \psi$, we are left to check $\fru\circ
\vphi=\frv$ and this follows from $\fru\circ \vphi \circ
\widetilde{\vphi}_i=\fru\circ \widehat{\vphi}_i=\vphi_i \circ \fru_i
=\frv\circ \widetilde{\vphi}_i$.

For the proof of \eqref{dirlimh2} it suffices to note that $(\Ker
\fru_i, \widehat f_{ji}|_{\Ker \fru_i})$ is a directed system and
that
$$
   0 \to \Ker \fru_i \to \uce(L_i) \to L_i \to 0
 $$
is exact for every $i\in I$. The claim then follows from the fact
that direct limits preserve exact sequences. \end{proof}

\begin{remark} \label{rem:S} Theorem~\ref{ucedirlim} is proven in \cite[App.]{S} for the
case $I=\mathbb{N}$ with the natural order and a directed system of
Lie algebras over algebraically closed fields of characteristic zero
$L_0 \xrightarrow{f_0}  L_1 \xrightarrow{f_1} \cdots$ satisfying the
condition that all $f_i$ are monomorphisms with $f_i\big(
\gz(L_i)\big) \subset \gz(L_{i+1})$.
\end{remark}

\begin{remark} 
We note that for Lie algebras the formula \eqref{dirlimh2} is not 
new. Indeed, by \cite[Cor.~7.3.6]{We} the homology groups of a Lie 
algebra can be interpreted as torsion groups for the universal 
enveloping algebra and it is known that torsion commutes with direct 
limits, see e.g., \cite[Cor. 2.6.17]{We}. The proof presented here 
is more direct and works in the super setting as well.
\end{remark}

For the next corollary we recall that a perfect Lie superalgebra is
called {\it centrally closed\/} if $\fru : \uce(L) \to L$ is an
isomorphism.

\begin{cor}\label{centrally closed} If $( L_i, f_{ji} )$ is a directed
system of perfect and centrally closed Lie superalgebras, then
$\dirlim L_i$ is perfect and centrally closed.
\end{cor}

\bigskip

\section{Examples:
 Universal central extensions of some infinite rank Lie superalgebras}
\label{sec:2}

In this section we will consider some examples of universal central
extensions of direct limits Lie superalgebras, mainly those which
are direct limits of some of the classical Lie superalgebras. In
order
to use the known results on their universal central extensions, we will in this section assume that all Lie superalgebras are
defined over a commutative, associative, unital ring $k$, rather
than an arbitrary base superring as in Section~\ref{sec:1}.

\begin{example}[Special linear Lie superalgebra $\lsl(I;A)$ for $A$ an associative
superalgebra] \label{exam:new} Let $I=I_{\bar 0} \cup I_{\bar 1}$ be
a superset, i.e., a partitioned set. Let $A$ be a unital
associative, but not necessarily commutative $k$-superalgebra. We
denote by $\Mat(I;A)$ the associative $k$-superalgebra whose
underlying module consists of $|I|\times |I|$-finitary matrices with entries
from $A$ (only finitely many non-zero entries) and $\ZZ_2$-grading given by $ |E_{ij}(a)| = |i| + |j| + |a|$. Here $E_{ij}(a)\in \Mat(I;A)$ has
entry $a$ at the position $(ij)$ and $0$ elsewhere. The product of
$\Mat(I;A)$ is the usual matrix multiplication. Clearly $\Mat(I;A)$
only depends on the cardinality of $I$ (and of course on $A$). For a
finite $I$ we put
$$
\Mat(m,n;A) := \Mat(I;A) \quad
   \hbox{if $|I_{\bar 0}| = m$ and $|I_{\bar 1}|=n$.}
$$
 A matrix $x\in \Mat(m,n;A)$ written as
\begin{equation} \label{exam:new2}
 x= \begin{array}{cc}
        \left[\begin{array}{cc} x_1 & x_2 \\ x_3 & x_4 \end{array} \right]
        \begin{array}{c} m \\ n \end{array} \\
        \begin{array}{c c c} m & n\quad & \end{array}
   \end{array}
\end{equation} is then even (resp. odd) if $x_1$ and $x_4$ are matrices with even
entries (resp. odd) entries and $x_2$ and $x_3$ are matrices with
odd (resp. even) entries.

We let $\lgl(I;A)$ be the Lie superalgebra associated to the
associative superalgebra $\Mat(I;A)$. Its product is $[x,y] = xy -
(-1)^{|x||y|}yx $. We assume $|I| \ge 3$ and then have that the Lie
superalgebra
$$
   \lsl(I;A) := [\lgl(I;A), \,  \lgl(I;A)]
$$
is perfect and satisfies $
 \lsl(I;A) = \{ x\in \Mat(I;A) : \str(x) \in [A,A]\}$.
Here the (super)trace $\str$ of a matrix $x=(x_{ij})\in \Mat(I;A)$
is given by $\str(x) = \sum_{i\in I_{\bar 0}} x_{ii} - \sum_{i\in
I_{\bar 1}} x_{ii}$, where $[A,A]$ is the span of all commutators
$[a_1, a_2] = a_1a_2 - (-1)^{|a_1||a_2|} a_2a_1$, $a_i \in A$. Observe that the Lie superalgebra
$\lsl(I;A)$ is generated by matrices $E_{ij}(a)$, $i\ne j \in I$,
$a\in A$, and these generators satisfy the relations
\begin{eqnarray}
  &a \mapsto E_{ij}(a) \quad \hbox{is $k$-linear;} \label{exam:new3}
  \\
  &[E_{ij}(a), E_{pq}(b)] = \delta_{jp} E_{iq}(ab) -
       (-1)^{|E_{ij}(a)||E_{pq}(b)|} \delta_{iq}E_{pj}(ba)
        \label{exam:new4}
\end{eqnarray}

For $I$ as above we define the (linear) {\em Steinberg Lie
superalgebra\/} $\st(I;A)$ as the Lie $k$-superalgebra presented by
generators $e_{ij}(a)$ with $i,j\in I, i\ne j$, $a\in A$ and
relations \eqref{exam:new3}--\eqref{exam:new4} with $E_{ij}(a)$
replaced by $e_{ij}(a)$. We then have a canonical Lie superalgebra
epimorphism
$$
  \frv_I : \st(I;A) \to \lsl(I;A), \quad e_{ij}(a) \mapsto E_{ij}(a).
$$
If $|I_{\bar 0}|=m$ and $|I_{\bar 1}| = n$, we put $\st(m,n;A) =
\st(I;A)$, $\lsl(m,n;A) = \lsl(I;A)$ and $\frv_{mn} : \st(m,n;A) \to
\lsl(m,n;A)$ for $v_I$.

We also need the first cyclic homology group $\HC_1(A)$. To define
it, we use
$$
      \ll A , A\gg \; = (A\ot_k A) / \mathcal{H}
$$ where $\mathcal{H}$ is the span of all elements of type
\begin{align*}
  a &\otimes b+ (-1)^{|a| |b|}  b\otimes a, \quad
     a_{\bar{0}} \otimes a_{\bar{0}} \quad \text{for } a_{\bar{0}}\in A_{\bar{0}}, \\
 (-1)^{|a||c|} a &\otimes bc +(-1)^{|b||a|} b\otimes
          ca +(-1)^{|c||b|} c\otimes ab
\end{align*}
for $a,b,c\in A$. We abbreviate $\ll a,b \gg \, = a\ot b +
\mathcal{H}$. Observe that there is a well-defined commutator map
$$
   \gc : \; \ll A, A\gg \; \to A, \quad \ll a,b\gg \,\mapsto [a,b].
$$
We put
$$
  \HC_1(A) = \Ker \gc = \{ \textstyle \sum_i \ll a_i, b_i \gg
   \;  : \sum_i [a_i, b_i] = 0 \}$$
We will use the following assumption:
\begin{equation}\label{exam:new5}
 \hbox{\it $\frv_F $ for $5\le |F|< \infty $ is a universal central
 extension with $\Ker \frv_F\cong \HC_1(A)$.}
\end{equation}
The assumption \eqref{exam:new5} is true in any one of the following
situations:
\begin{itemize}
  \item[(a)] $n=0$, $A$ an algebra \cite{KL} or a superalgebra \cite{CG},

  \item[(b)] $A$ an algebra \cite{MP1, IK}. \end{itemize}
The references \cite{CG} and \cite{IK} assume that $k$ is a
commutative ring containing $\frac{1}{2}$ and that the underlying
module of $A$ is free with a basis containing the identity element
of $A$. We note that \eqref{exam:new5} is not true for $m+n \le 4$, see the
papers \cite{CG, G, GS, G2} which deal with the case $3\le m+n \le
4$.\end{example}

\begin{prop}\label{newassprop}  Assume\/ \eqref{exam:new5} holds and $I$ is a (possibly infinite) set with
$|I|\ge 5$. Then $\frv_I : \st(I;A) \to \lsl(I;A)$ is a universal
central extension with kernel isomorphic to $\HC_1(A)$.
\end{prop}

\begin{proof} This can be proven by adapting the proof of \eqref{exam:new5}
to our setting. Instead we prefer to give a proof based on
Theorem~\ref{ucedirlim}. This is possible since, denoting by $\scF$
the set of finite subsets of $I$ ordered by inclusion, the Lie
superalgebra $\lsl(I;A)$ is indeed a direct limit: $ \lsl(I;A) =
\bigcup_{F\in \scF} \lsl(F;A) \cong \dirlim_{F\in \scF} \lsl(F;A)$.
Hence $\uce(\lsl(I;A)) \cong \dirlim_{F\in \scF} \uce(\lsl(F;A))
\cong \dirlim_{F\in \scF} \st(F;A)$. Thus we need to show that
$\dirlim_{F\in \scF} \st(F;A) \cong \st(I;A)$. This follows from the
diagram below, where $\psi_F$ is given by sending a generator
$e_{ij}(a) \in \st(F;A)$ to $e_{ij}(a) \in \st(I;A)$. The existence
of $\vphi$ then follows from the definition of a direct limit,
applied to $(\psi_F; F\in \scF)$. The families  $(e_{ij}(a)\in
\st(F;A): F\in \scF)$ give rise to elements $\fre_{ij}(a) \in
\dirlim_{F\in \scF} \st(F;A)$ satisfying the relations
\eqref{exam:new3}--\eqref{exam:new4}, whence the existence of the
map $\psi$ sending $e_{ij}(a)\in \st(I;A)$ to $\fre_{ij}(a)$.
 \begin{equation*} \xymatrix{
 \st(F;A)  \ar[rr]^{f_{F'F}}
    \ar[dr]^{\vphi_F}
    \ar@/_/[ddr]_{\psi_{F}}
    & & \st(F';A) \ar[dl]_{\vphi_{F'}}
        \ar@/^/[ddl]^{\psi_{F'}}\\
  &\dirlim \st(F;A)  \ar@<0.5ex>[d]^{\varphi}& \\
& \st(I;A) \ar@<0.5ex>[u]^\psi&}  \end{equation*}  It is immediate
that $\vphi$ and $\psi$ are inverses of each other, and that $\Ker
\frv_I \cong \HC_1(A)$.
 \end{proof}

\begin{example}[$\lsl(I;A)$ for $A$ an associative commutative
superalgebra] \label{sl} Let $A$ be a unital associative and
commutative $k$-superalgebra, thus $[A,A]=0$. Therefore the
descriptions of $\lsl(I;A)$ and $\HC_1(A)$ simplify to
\begin{align*}
 \lsl(I;A) &= \{ x\in \lgl(I;A) : \str(x)=0 \} \cong
                  \lsl(I;k) \ot_k A, \\
  \HC_1(A) &= \, \ll A, A \gg.
\end{align*}
Moreover, the universal central extension $\st(I;A)$ can be
described via a $2$-cocycle as follows. The Lie superalgebra $\lsl(m,n;A)$ has a central
$2$-cocycle $\tau_{mn}$ with values in $\HC_1(A)$:
$$
   \tau_{mn}(x,y) =
\allowbreak \textstyle \sum_{1\leq i\leq m, 1\leq j\leq m+n}\ll
x_{ij},y_{ji}\gg- \sum_{m+1\leq i\leq m+n, 1\leq j\leq m+n}\ll
x_{ij},y_{ji}\gg$$ for $x=(x_{ij})$, $y=(y_{ij}) \in \lsl(m,n;A)$.
We let $\lsl(m,n,A) \oplus \HC_1(A)$ be the corresponding Lie
superalgebra, and view it as a central extension of $\lsl(m,n;A)$ by
projecting onto the first factor. From now on we suppose $m+n \ge 5$
and that the map $$h_{mn} : \uce(\lsl(m,n;A)) \to \lsl(m,n;A) \oplus
\HC_1(A), \quad h_{mn} \la x,y \ra = [x,y] \oplus \tau_{mn}(x,y)$$
is an isomorphism of central extensions:
\begin{equation} \label{sl1}
 \uce\big(\lsl(m,n;A)\big) \cong \lsl(m,n;A) \oplus \HC_1(A)
 \quad \hbox{\it as central extensions.}
\end{equation}
The assumption (\ref{sl1}) is true in any one of the following
situations:
\begin{itemize}
  \item[(a)] $n=0$, $A$ an algebra \cite{KL} or a superalgebra \cite{CG},

  \item[(b)] $A$ an algebra \cite{MP1, IK}. \end{itemize}

Let now $(\mathfrak{sl}(m_i,n_i,A),f_{ji})$ be a directed system of
Lie superalgebras with $m_i + n_i \ge 5$. The transition maps
$f_{ji}:\mathfrak{sl}(m_i,n_i,A)\rightarrow
\mathfrak{sl}(m_j,n_j,A)$ lift uniquely to Lie superalgebra
morphisms
$\widehat{f}_{ji}:\uce(\mathfrak{sl}(m_i,n_i,A))\rightarrow
\uce(\mathfrak{sl}(m_j,n_j,A))$. Hence, we get a directed system
$(\lsl(m_i,n_i,A) \oplus \HC_1(A))_{i \in I}$ with transition maps $
h_{m_j\, n_j} \circ \widehat f_{ji} \circ h_{m_i\, n_i}^{-1} =
f_{ji} \oplus g_{ji}$ where the map $g_{ji}: \HC_1(A)\rightarrow
\HC_1(A)$ is given by $g_{ji}(\tau_{m_i n_i}(x,y))=\tau_{m_j
n_j}(f_{ji}(x),f_{ji}(y))$ for $x,y\in \mathfrak{sl}(m_i,n_i,A)$. We now define a central $2$-cocycle $\tau_I$ for the direct limit
Lie superalgebra
$$\mathfrak{sl}(I;A)=\dirlim \mathfrak{sl}(m_i,n_i,A)$$
with values in $\HC_1(A)$.  Let $x,y\in\mathfrak{sl}(I;A)$. Thus
$x=\varphi_p(x_p)$ and $y=\varphi_q(y_q)$ for some $p,q\in I$,
$x_p\in \lsl(m_p,n_p, A)$ and $y_q\in \lsl(m_q,n_q, A)$. Here
$\varphi_p,\varphi_q$ are the canonical maps for
$\mathfrak{sl}(I;A)$. There exists $k\in I$ such that $k\geq p,
k\geq q$ and $f_{kp}(x_p), f_{kq}(y_q)\in  \lsl(m_k,n_k, A)$. Then
the cocycle $\ta_I$ for $\mathfrak{sl}(I;A)$ is given by
\begin{equation} \label{deftau} \tau_I (x,y)
=\tau_{m_k n_k}(f_{kp}(x_p), f_{kq}(y_q)).\end{equation} We now get from Proposition \ref{newassprop} that
$\mathfrak{st}(I;A)\cong \uce\big( \mathfrak{sl} (I;A)\big)
\cong\mathfrak{sl}(I;A) \oplus \HC_1(A),$ where the $2$-cocycle
$\tau_I$ is given explicitly by \eqref{deftau}. Summarizing the
above, we have proven the following.
\end{example}

\begin{cor}\label{sl2cor} Let $A$ be a unital associative commutative superalgebra over a commutative ring $k$. Let $(I, \leq)$ be an arbitrary directed
set and let $(\mathfrak{sl}(m_i,n_i,A),f_{ji})$ be a directed system
of Lie superalgebras with $m_i + n_i \ge 5$. We suppose \eqref{sl1}
and denote by $\mathfrak{sl}(I;A):=\dirlim \mathfrak{sl}(m_i,n_i,A)$
the corresponding direct limit, which is a perfect Lie superalgebra
of possibly infinite rank. Then
\begin{equation} \label{slscor1}  \uce\big( \mathfrak{sl}
(I;A)\big) \cong \mathfrak{sl}(I;A) \oplus \HC_1(A) \end{equation}
as central extensions, where the Lie superalgebra structure on the
right is given by the $2$-cocycle $\tau_I$ of \eqref{deftau}.
\end{cor}

\begin{example}[$\lsl_J(A)$ for $A$ an associative
algebra] \label{exam:slinf} Let $A$ be an associative unital
$k$-algebra over a commutative ring $k$ containing $\frac{1}{2}$,
and let $J$ be an arbitrary, possible infinite set with $|J|\ge 5$.
We denote by $\lsl_J(A)$ the Lie algebra of finitary matrices over
$A$ (only finitely many non-zero entries) and with trace in $[A,A]$.
Since $\lsl_J(A)$ is the direct limit of the Lie algebras
$\lsl_F(A)$ where $F$ runs through the finite subsets of $J$,
Corollary~\ref{sl2cor} implies that $\uce(\lsl_J(A)) \cong \lsl_J(A)
\oplus \HC_1(A)$. This is proven in \cite{KL} for $J$ finite or
countable and in \cite{Wel} for arbitrary $J$, using the theory of
root graded Lie algebras.
\end{example}

\begin{example}[Ortho-symplectic Lie superalgebra $\osp(I;A)$ for $A$ an associative commutative
superalgebra] \label{exam:osp} The ortho-symplectic Lie superalgebra
$\osp(m,n;A)$ can be defined in the usual way, see for example
\cite{IK,IK2,MP2}. Since $\osp(m,n;A)$ is a subalgebra of
$\lsl(m,n;A)$, the restriction of the $2$-cocycle $\ta_{mn}$ of
Example~\ref{sl} defines a $2$-cocycle of $\osp(m,n;A)$ with values
in $\HC_1(A)$ and thus gives rise to a central extension. We suppose
that the map $\uce(\osp(m,n;A)) \to \osp(m,n;A) \oplus \HC_1(A)$,
given by $\la x,y\ra \mapsto [x,y] \oplus \tau_{mn}(x,y)$ is an
isomorphism:
\begin{equation} \label{exam:osp1}
 \uce\big(\osp(m,n;A)\big) \cong \osp(m,n;A) \oplus \HC_1(A)
 \quad \hbox{\it as central extensions.}
\end{equation}
Our assumption \eqref{exam:osp1} is fulfilled in any one of the
following situations:
\begin{itemize}
  \item[(a)] $k$ a field of characteristic $0$ \cite{IK2},

  \item[(b)] $A$ a commutative algebra \cite{IK, MP2}.
\end{itemize}
The reference \cite{MP2} assumes that $m\ge 5,n\ge 10$. 

Let now $(\osp(m_i,n_i;A),f_{ji})$ be a directed system of Lie
superalgebras. One shows as in Example~\ref{sl} that there exists a
well-defined $2$-cocycle $\ta'_I$ for the direct limit Lie
superalgebra
$$\osp(I;A):=\dirlim \osp(m_i,n_i;A)$$
with values in $\HC_1(A)$. From
Theorem \ref{ucedirlim} we then get as in Corollary~\ref{sl2cor}
\begin{equation} \label{exam:osp2} \begin{split}
\uce\big( \osp(I;A)\big) &\cong \dirlim
\uce\big(\osp(m_i,n_i;A)\big) \cong \dirlim
\big(\osp(m_i,n_i;A)\oplus \HC_{1}(A)\big) \\ & \cong\osp(I;A)
\oplus \HC_1(A).
\end{split}\end{equation}
\end{example}

\begin{example}[Locally finite Lie superalgebras] \label{exam:lfls}
Classically semisimple locally finite Lie superalgebras over
algebraically closed fields of characteristic $0$ were introduced
and studied in \cite{penkov:2004}, including a classification of the
simple infinite dimensional ones which admit a local system of root
injections of classical finite dimensional Lie superalgebras. They
are all direct limits $L=\dirlim_i L_i$ of classical simple Lie
superalgebras $L_i$, $i\in \N$ with $f_i : L_i \to L_{i+1}$ being
the natural inclusions. Referring the reader to \cite{penkov:2004}
for details, we simple present the classification list. We
abbreviate $\lsl(m,n)=\lsl(m,n;k)$ and $\osp(m,n) = \osp(m,n;k)$ in
the notation of \ref{sl} and \ref{exam:osp} respectively. The Lie
superalgebra ${\rm SP}(m)$ is the subalgebra of $\lsl(m,m)$ which
leaves invariant an odd nondegenerate super-antisymmetric bilinear
form, and $\sq(m)$ is the subalgebra of $\lsl(m,m)$ consisting of
matrices of the form (\ref{exam:new2}) with $x_1=x_4$, $x_2=x_3$ and
$\tr(x_2)=0$. With these notations, an infinite dimensional simple
Lie superalgebra, which admits a local system of root injections of
classical finite dimensional Lie superalgebras $L_i$, is isomorphic
to a Lie superalgebra $L$ in the following table ($i\geq 2,n\ge
0,k\ge 0,r\ge0, m\ge 2$):
\begin{equation} \label{exam:lflstab}
 \begin{tabular}{|c c || c c|}
\hline
 $L$ & $L_i$ & $L$ & $L_i$ \\
\hline
 $\lsl(\infty,n)$ & $\lsl(i,n)$ & $\lsl(\infty, \infty)$ &
     $\lsl(i,i)$\\
 ${\rm B}(\infty, 2k)$ & $\osp(2i+1, 2k)$ &
          $\quad {\rm B}(\infty, \infty)\quad $ & $\quad \osp(2i+1, 2i)\quad $ \\
 $\quad {\rm B}(2r+1,\infty)$  \quad & $\quad \osp(2r+1, 2i) \quad$ &
   ${\rm C}(\infty)$& \quad $\osp(2,2i)$ \\
 ${\rm D}(\infty, 2k)$ & $\osp(2i, 2k)$ &
    ${\rm D}(\infty, \infty)$ & $\osp(2i, 2i)$ \\
 ${\rm D}(2m,\infty)$ & $\osp(2m,2i)$ &
   ${\rm SP}(\infty)$ & ${\rm SP}(i)$ \\
 $\sq(\infty)$ & $\sq(i)$ & & \\
\hline
\end{tabular}
 \end{equation}
\end{example}

\begin{cor} \label{corlfls} The Lie superalgebras $L$ listed in
table \eqref{exam:lflstab} are all centrally closed.
\end{cor}

\begin{proof} The Lie superalgebras $L_i$ in table \eqref{exam:lflstab}
are all perfect. In view of Theorem~\ref{ucedirlim} it therefore
remains to show that they are centrally closed for large $i$. For
$L_i$ of type $\lsl(m,n)$ or $\osp(m,n)$ this follows from
\eqref{sl1} and \eqref{exam:osp1} since $\HC_1(k) = \{0\}$. For the
remaining two types this follows from \cite[Th.~5.10]{IK2}.
\end{proof}

\begin{example}[Locally finite Lie algebras] \label{exam:lf}
Corollary~\ref{corlfls} applies in particular to the simple locally
finite Lie algebras $\lsl(\infty) = \lsl(\infty, 0)$, $
\mathfrak{o}(\infty) = {\rm B}(\infty, 0) = {\rm D}(\infty, 0)$,
   $\frsp(\infty) = {\rm D}(0, \infty) $ (the only infinite dimensional simple root reductive Lie
algebras), studied for example in \cite{BB}, \cite{DP} and
\cite{PS}.
\end{example}

\begin{example}[Root-graded Lie algebras]\label{exam:rg}
(a) Let $k$ be a ring in which $2$ and $3$ are invertible, and let
$L$ be a Lie algebra graded by a locally finite reduced root system
$R$ as defined in \cite{n:3g}, see also \cite[5.1]{n:persp}. Thus,
$L$ is graded by $\scQ(R)=\Span_\Z(R)$ with $\supp_{\scQ(R)}L = R$,
i.e.,
$$ L = \textstyle  \bigoplus_{\al \in R} L_\al , \quad
   [L_\al, L_\be] \subset L_{\al + \be}, $$
satisfies $L_0 = \sum_{0\ne \al \in R} [L_\al, L_{-\al}]$ and has
the property that for $\al \ne 0$ there exists an $\lsl_2$-triple
$(e_\al, h_\al, f_\al) \in L_\al \times L_0 \times L_{-\al}$ such
that $[h_\al, x_\be] = \langle \be, \al\ch\rangle x_\be$ holds for
every $\be \in R$ and $x_\be \in L_\be$. By \cite[3.15]{lfrs}, $R$
is a direct limit of finite root systems, say $R=\dirlim R_i$ where
$i$ runs through a directed set $(I, \le)$. The subalgebra
\begin{equation} \label{exam:rg1}
     L_i = \textstyle \big( \bigoplus_{0\ne \al \in R_i} L_\al \big)
         \oplus \sum_{0\ne \al\in R_i} [L_\al, L_{-\al}]
\end{equation} is graded by the root system $R_i$, and it is immediate that
$L= \dirlim L_i$.

Any root-graded Lie algebra is perfect, whence $\uce(L) \cong
\dirlim \uce(L_i)$ by Theorem~\ref{ucedirlim}. In fact, something
more precise is true. One knows that the universal central extension
of a root-graded Lie algebra is again graded by the same locally
finite root system, \cite[Prop.~5.4]{n:persp}. Thus the root-graded
Lie algebra $\uce(L)$ is a direct limit of root-graded Lie algebras,
$$\uce(L) \cong  \dirlim \uce(L)_i,$$ where $\uce(L)_i$ is defined in the
same way as $L_i$.

(b) Suppose in the following that $k$ is a field of characteristic
zero. If $K$ is a centreless Lie algebra  graded by a finite
irreducible reduced root system, the group $\rmH_2(K)$ is known to
be the full skew-dihedral homology group $\HF(\mathfrak{a})$, where
$\mathfrak{a}$ is the coordinate algebra of $K$,
\cite[Th.~4.13]{ABG}.

Let now $L$ be a centreless Lie algebra graded by a locally finite
irreducible reduced root system $R$ of rank $\ge 9$. Then $R=\dirlim
R_i$ where the $R_i$ are finite, irreducible, reduced, have rank
$\ge 9$ and are of the same type as $R$, \cite[8.3]{lfrs}. It is
moreover no harm to assume that $R_0 \subseteq R_i$ for some fixed
$0\in I$. It then follows that the root-graded Lie algebras $L_i$ of
\eqref{exam:rg1} all have the same coordinate algebra $\mathfrak{a}$
(this has also been noted by M.~Yousofzadeh). Notice that if $f_{ji}(\gz(L_i))\subset \gz(L_j)$, then $\dirlim L_i  \cong \dirlim L_i/ \gz(L_i)$. Hence
\begin{equation}\begin{split} \label{exam:rg2}
   \uce(L) \cong \uce \big(\dirlim L_i/ \gz(L_i)\big) &\cong \dirlim \uce (L_i/ \gz(L_i) )\\ &\cong \dirlim L_i/ \gz(L_i) \oplus \HF(\mathfrak{a})  \cong L \oplus \HF(\mathfrak{a}).
\end{split}\end{equation} Special cases of root-graded Lie algebras are the so-called Lie
tori, which occur as cores and centreless cores of locally extended
affine Lie algebras \cite{morita-yoshii,neeb:LEALA}. More generally,
the cores of affine reflection Lie algebras are root-graded Lie
algebras of possibly infinite rank \cite[6.4, 6.5]{n:persp}.
\end{example}

\end{document}